\documentclass[11pt,reqno]{amsart}
\usepackage{amsmath} 
\usepackage[T1]{fontenc}

\usepackage[numbers,sort&compress]{natbib}
\usepackage[pdfusetitle,colorlinks=true,
  pdfauthor={Hao Shen, Jiaqi Wang, and Lihong Zhi},
  linkcolor=blue!45!black,citecolor=blue!45!black,
  urlcolor=blue!45!black]{hyperref}
  \usepackage{amssymb}
  
\usepackage{mathtools}
\usepackage{xcolor}
\theoremstyle{plain}
\newtheorem{theorem}{Theorem}
\newtheorem{lemma}[theorem]{Lemma}
\newtheorem{proposition}[theorem]{Proposition}

\newcommand{\red}[1]{{#1}}
\newcommand{\blue}[1]{{#1}}
\newcommand{\DD}{\mathsf D}
\newcommand{\ord}{\operatorname{ord}}
\newcommand{\rank}{\operatorname{rank}}
\newcommand{\Frac}{\operatorname{Frac}}
\newcommand{\ab}{\mathrm{ab}}
\usepackage{listings}
\usepackage{textcomp}
\definecolor{keywordcolor}{rgb}{0.65,0.08,0.08}
\definecolor{symbolcolor}{rgb}{0.00,0.10,0.55}
\definecolor{sortcolor}{rgb}{0.05,0.42,0.05}
\definecolor{commentcolor}{rgb}{0.45,0.45,0.45}
\providecolor{keywordcolor}{rgb}{0.7,0.1,0.1}
\providecolor{tacticcolor}{rgb}{0.0,0.1,0.6}
\providecolor{commentcolor}{rgb}{0.4,0.4,0.4}
\providecolor{symbolcolor}{rgb}{0.0,0.1,0.6}
\providecolor{sortcolor}{rgb}{0.1,0.5,0.1}
\providecolor{attributecolor}{rgb}{0.7,0.1,0.1}
\providecolor{stringcolor}{rgb}{0.5,0.3,0.0}
\providecolor{errorcolor}{rgb}{1.0,0.0,0.0}

\lstnewenvironment{leancode}{\lstset{language=lean,abovecaptionskip=-\medskipamount}}{}
\newcommand{\Sf}{\mathcal S}
\newcommand{\Af}{\mathcal A}
\title{On the asymptotic Makar-Limanov rank conjecture}

    	\author[]{Sizhuo Yan }
	\address{ State Key Laboratory of Mathematical Sciences, Academy of Mathematics and Systems Science, Chinese Academy of Sciences, Beijing 100190, China }

	\email{yansizhuo@amss.ac.cn}

      	\author[]{Hao Shen}
	\address{ State Key Laboratory of Mathematical Sciences, Academy of Mathematics and Systems Science, Chinese Academy of Sciences, Beijing 100190, China } 
	\email{shenhao24@amss.ac.cn}

    	\author[]{Jianting Yang}
	\address{ School of Mathematical Sciences, University of Electronic Science and Technology of China, Chengdu 611731,  China} 
	\email{jtyangsdp@gmail.com} 

\date{29 September 2026}

\begin{document}
\maketitle
\begin{abstract}
Let $\Bbbk$ be an algebraically closed field of characteristic zero and let $f$ be a nonconstant polynomial in finitely many freely noncommuting variables. We prove the asymptotic Makar-Limanov rank conjecture, namely that the normalized rank of a value of $f$ can be made arbitrarily small by choosing matrices over $\Bbbk$ of finite size. 
\end{abstract}

\red{
\section{Introduction}\label{sec:introduction}
Let $\Bbbk$ be an algebraically closed field of characteristic zero, and let
$\Bbbk\langle x\rangle=\Bbbk\langle x_1,\ldots,x_d\rangle$ be the free unital associative algebra in a finite tuple of freely noncommuting variables. A free polynomial can be evaluated on tuples of matrices of any common size: the order of the letters is preserved, and a scalar $c$ is sent to $cI_n$. Studying these evaluations simultaneously over all matrix sizes leads to dimension-free matricial Nullstellens\"atze, which relate properties of matrix values to the algebraic structure of the free algebra. As surveyed by \citet{zbMATH08130993}, the resulting correspondence depends on whether one considers zero values, common kernels, singular matrix values, or vanishing traces. 

The trace condition illustrates the importance of choosing an appropriate notion of vanishing. \citet[Corollaries~4.7 and~4.8]{zbMATH05613896} show that a free polynomial has trace zero on all tuples of a fixed matrix size $n\geq2$ precisely when it is a sum of commutators and a polynomial identity for that size; vanishing trace at every size is equivalent to being a sum of commutators. In contrast, even a nonconstant free polynomial need not have a matrix value equal to zero. Indeed, with $[a,b]=ab-ba$, the polynomial $1-[x_1,x_2]$ has no such value in characteristic zero, since
\[
 \operatorname{tr}(I_n-[X,Y])=n\ne0
 \qquad\bigl(X,Y\in M_n(\Bbbk)\bigr).
\]
Thus exact matrix zeros alone cannot play the same role as scalar zeros of commutative polynomials over an algebraically closed field.

The obstruction disappears for this example if vanishing is measured by normalized rank. Let $E_{ij}$ denote the standard matrix units and set
\[
 X_n=\sum_{j=1}^{n-1}jE_{j,j+1},\qquad
 Y_n=\sum_{j=1}^{n-1}E_{j+1,j}.
\]
\blue{Direct calculation gives}
\[
 I_n-[X_n,Y_n]=nE_{nn},\qquad
 \frac{\rank_{\Bbbk}(I_n-[X_n,Y_n])}{n}=\frac1n.
\]
Consequently the normalized rank tends to zero although an exact matrix zero is impossible. This example is discussed in \citet[Section~6]{zbMATH08130993} as motivation for replacing exact zeros by values of arbitrarily small normalized rank.

\blue{Related earlier work includes} Makar-Limanov's construction of algebraically closed skew fields \citep{zbMATH03889737}. That work develops a setting for solving noncommutative polynomial equations and applies it to the Freiheitssatz for associative algebras. \blue{Makar-Limanov conjectured that every nonconstant free polynomial has matrix values of arbitrarily small normalized rank.} \citet[Conjecture~6.1]{zbMATH08130993} records this assertion, attributing it to a 2017 talk at Bar-Ilan University and noting that it may have been proposed earlier. The purpose of this paper is to prove this asymptotic Makar-Limanov rank conjecture.

\begin{theorem}\label{thm:main}
For every algebraically closed field $\Bbbk$ of characteristic zero, every finite tuple $x=(x_1,\ldots,x_d)$ of freely noncommuting variables, and every
$f\in\Bbbk\langle x\rangle\setminus\Bbbk$, define
\[
 r_f(n)=\min_{Z_1,\ldots,Z_d\in M_n(\Bbbk)}\rank_{\Bbbk} f(Z_1,\ldots,Z_d)
 \qquad(n\geq1).
\]
Then
\[
 \inf_{n\geq1}\frac{r_f(n)}n=0.
\]
\end{theorem}

For each $n$, the minimum exists because the set of possible ranks is a nonempty subset of $\{0,\ldots,n\}$; the ratios and their infimum are taken in $\mathbb R$. Equivalently, for every real $\varepsilon>0$ there are a positive integer $n$ and matrices $Z_1,\ldots,Z_d\in M_n(\Bbbk)$ such that $\rank_{\Bbbk}f(Z_1,\ldots,Z_d)<\varepsilon n$. The theorem establishes exactly the assertion of \citet[Conjecture~6.1]{zbMATH08130993}. The conclusion permits positive rank at every finite size.  A separate question in \citet[Section~6]{zbMATH08130993} asks whether $r_f(n)\leq1$ whenever $f$ is not constant on $M_n(\Bbbk)^d$.

\blue{
  Inspired by \citep[Example~3.2]{zbMATH07292439}, we construct a counterexample already for $n=3$. Set
\[
 p(x,y)=[x,[y,x][x,[y,x]^2]],\qquad f(x,y)=1+p(x,y).
\]
For arbitrary $X,Y\in M_3(\Bbbk)$, write $P=p(X,Y)$. Then
\[
 \operatorname{tr}(P)=\operatorname{tr}(P^2)=0.
\]
The first identity holds because $p$ is a commutator. The second can be verified by a symbolic computation\footnote{\blue{The  source code used for this symbolic verification is available at \url{https://github.com/MechMath/MakarLimanov-formalization/blob/main}.}}, which implies  $P^3=(\det P)I_3$. Thus the eigenvalues of $P$ are $c,c\omega,c\omega^2$ for some $c\in\Bbbk$, where $\omega$ is a primitive cube root of unity. At most one of these eigenvalues equals $-1$, so $\rank(I_3+P)\geq2$.

Moreover, take
\[
 X_0=\begin{pmatrix}0&0&0\\0&\tfrac12&0\\0&0&1\end{pmatrix},\qquad
 Y_0=\begin{pmatrix}0&2&2\\2&0&4\\2&2&0\end{pmatrix}.
\]
Direct calculation yields
\[
 p(X_0,Y_0)=\begin{pmatrix}0&1&-1\\1&0&-1\\2&-1&0\end{pmatrix},\qquad
 \rank\bigl(I_3+p(X_0,Y_0)\bigr)=2.
\]
Furthermore, by $p(0,0)=0$, the polynomial $f$ is not constant on $M_3(\Bbbk)^2$, which shows $r_f(3)=2$.
}

\blue{The related correspondence between rank inequalities and membership in two-sided ideals is formulated in \citet[Conjecture~6.2]{zbMATH08130993}. \citet{arXiv:2610.01804} prove this conjecture when the generators are homogeneous and give a counterexample in the nonhomogeneous case.}

The proof proceeds by constructing solutions to increasing finite precision and then converting them into matrix values. Section~\ref{sec:reduction} reduces the problem, unless a scalar zero already exists, to a two-variable polynomial $h=1+g$ whose nonconstant part has zero abelianization. Section~\ref{sec:symbol} introduces a formal symbol algebra with two commuting coefficient derivations and proves the nonvanishing and variation estimates needed for Newton corrections. Section~\ref{sec:branch} realizes the differential equations for the leading coefficients while preserving specified inequations. In Section~\ref{sec:newton}, an orbit estimate for irreducible factors under roots of unity bounds the denominators of the correction exponents and ensures termination at every prescribed precision. Section~\ref{sec:matrix} uses a Taylor map and normal ordering to realize the symbols as infinite matrices with controlled support; finite compression confines the error to boundary rows and columns. Finally, Section~\ref{sec:final} descends the resulting finite rank conditions to $\Bbbk$ and completes the proof.

 Throughout the paper, the degree of a word is its length; in particular, the empty word $1$ has degree zero.
}

\section{Reduction to two generators}\label{sec:reduction}
For a free polynomial $f$, write $f^{\ab}$ for its image in the commutative polynomial ring under abelianization.

\begin{lemma}\label{lem:reduction}
Let ${\red{\Bbbk}},d,f$ be as in Theorem~\ref{thm:main}, with $d\geq1$. Either $f$ has a scalar zero, or there are $c\in {\red{\Bbbk}}^\times$ and a nonconstant polynomial
\[
 h=c^{-1}f(xy,xy^2,\ldots,xy^d)=1+g\in {\red{\Bbbk}}\langle x,y\rangle,
 \qquad g\ne0,\quad g^{\ab}=0.
\]
The substitution $x_i\mapsto xy^i$ is injective. Every word in the support of $g$ contains both letters. If $D=\deg h$ and $A,M$ are the maximum numbers of $x$'s and $y$'s, respectively, in such a word, then
\[
 1\leq A,M\leq D-1.
\]
For all $X,Y\in M_n({\red{\Bbbk}})$,
\[
 \rank h(X,Y)=\rank f(XY,XY^2,\ldots,XY^d).
\]
\end{lemma}
\begin{proof}
The field is infinite, since characteristic zero makes its nonnegative integer multiples of $1$ distinct. A nonzero univariate polynomial of degree $m$ has at most $m$ roots: division by $T-a$ at a root and induction on degree prove this. Consequently every nonzero polynomial over ${\red{\Bbbk}}$ in finitely many variables has a nonzero value. Indeed, write it as a polynomial in the last variable, choose the other variables so that its nonzero leading coefficient remains nonzero by induction, and then choose the last variable outside the finite root set. The induction starts with a nonzero constant in zero variables.

Put $F=f^{\ab}$. Scalar evaluation of each word agrees with evaluation of its abelianization, so $f(a)=F(a)$ for $a\in {\red{\Bbbk}}^d$. If $F=0$, take $a=0$. If $F$ is nonconstant, choose a variable of positive degree and specialize the others so its leading coefficient is nonzero, using the preceding fact. The resulting nonconstant univariate polynomial has a root by algebraic closure. This gives a scalar zero of $f$. It also gives a zero value in every matrix size by substituting $a_iI_n$.

The remaining case is $F=c\ne0$. Define $\phi(x_i)=xy^i$. The image of a nonempty word $x_{i_1}\cdots x_{i_\ell}$ is
$xy^{i_1}\cdots xy^{i_\ell}$. Each $x$ starts a block, and the lengths of the following positive runs of $y$ recover the ordered indices $i_1,\ldots,i_\ell$. Distinct source words therefore have distinct images, and no nonempty word maps to $1$. Linear independence of words proves injectivity.

The commutative substitution $t_i\mapsto uv^i$ commutes with abelianization, as is checked on every word. Hence $\phi(f)^{\ab}=c$. Abelianization preserves the constant coefficient, so $f-c$ is a nonzero polynomial with zero constant coefficient. Set $g=c^{-1}\phi(f-c)$ and $h=c^{-1}\phi(f)$. Then $g\ne0$, $h=1+g$, and $g^{\ab}=0$. Every image word of $f-c$ has at least one $x$ and at least one $y$, and injectivity prevents cancellation between distinct words. Its length is the sum of these two positive counts. Taking maxima over the finite nonempty support proves the asserted bounds. In particular $D$ is finite; explicitly $D\leq(d+1)\deg f$.

Evaluation of a word after $\phi$ is its evaluation on $(XY,XY^2,\ldots,XY^d)$. Thus $h(X,Y)=c^{-1}f(XY,\ldots,XY^d)$. Multiplication of a linear map by a nonzero scalar preserves its image dimension, giving the rank identity. This transfers witnesses and requires no surjectivity of the substitution on matrix tuples.
\end{proof}
For the symbolic construction fix such a normalized $h=1+g$, and retain $D,A,M$. Write $g=\sum_{k=1}^M h_k$, where $h_k$ is the sum of words containing exactly $k$ copies of $y$; some $h_k$ may be zero, but $h_M\ne0$.

\section{The symbol algebra and its variations}\label{sec:symbol}
Let ${\red{\mathbb{F}}}/{\red{\Bbbk}}$ be a field equipped with commuting derivations $\delta,\eta$ that vanish on ${\red{\Bbbk}}$. For an integer $p\geq1$, put $L_p=p^{-1}\mathbb Z$. Define $\Sf_p({\red{\mathbb{F}}})$ to consist of formal sums
\[
 U=\sum_{r\in L_p}a_r\DD^r
\]
whose support is bounded above. Set $\ord0=-\infty$ and otherwise let $\ord U$ be the largest occupied exponent. Initially juxtaposition denotes the ordinary commutative multiplication of these series. Define
\[
 \Delta(a\DD^r)=(ra+\eta(a))\DD^{r-1},\qquad
 \delta(a\DD^r)=\delta(a)\DD^r,
 \qquad
 U\star V=\sum_{j\geq0}\frac{\Delta^j(U)\delta^j(V)}{j!}.
 \tag{3.1}\label{eq:star}
\]
Extend $\eta$ coefficientwise as well. A differential indeterminate $v$ over ${\red{\mathbb{F}}}$ means the field ${\red{\mathbb{F}}}(v_{ij}:i,j\geq0)$ on algebraically independent variables, with $v=v_{00}$, $\delta v_{ij}=v_{i+1,j}$ and $\eta v_{ij}=v_{i,j+1}$. These rules commute and extend to fractions by the quotient rule. We write ${\red{\mathbb{F}}}\{v\}={\red{\mathbb{F}}}[v_{ij}:i,j\geq0]$. Ordinary degree in the jets and total differential order $\max(i+j)$ are different notions.

\begin{lemma}\label{lem:algebra}
Formula~\eqref{eq:star} defines a unital associative ${\red{\Bbbk}}$-algebra, with ${\red{\Bbbk}}$ central and
$\ord(U\star V)\leq\ord U+\ord V$. The coefficient derivations act as derivations for $\star$, and
$[\DD,U]_\star=\delta U$. Compatible differential-field embeddings and inclusions into a multiple denominator lattice preserve the operations and orders.
\end{lemma}
\begin{proof}
If $U,V$ have exponent bounds $R,S$, an ordinary coefficient at $q$ uses $r+s=q$, hence $q-S\leq r\leq R$. This is a finite set in $L_p$. The analogous triple count proves associativity of the ordinary product. The Leibniz rule for $\eta$ gives
\[
 \Delta(ab\DD^{r+s})=((r+s)ab+\eta(a)b+a\eta(b))\DD^{r+s-1},
\]
so $\Delta$ is an ordinary-product derivation lowering order by at least one. The coefficient derivations do not increase order, and they commute with $\Delta$ because $\delta\eta=\eta\delta$.

A coefficient of $U\star V$ at $q$ uses $r+s-j=q$ with $j\geq0$, $r\leq R$, $s\leq S$. Thus $j\leq R+S-q$, and for fixed $j$ the exponent $r$ lies in $[q+j-S,R]\cap L_p$. Every coefficient is finite and every output exponent is at most $R+S$.

Expanding the Leibniz rules gives
\begin{align*}
 (U\star V)\star W
 &=\sum_{i,a,b\geq0}
 \frac{\Delta^{i+a}U\,\Delta^b\delta^iV\,\delta^{a+b}W}{i!a!b!},\\
 U\star(V\star W)
 &=\sum_{\substack{l,m,j\geq0\\j\leq l}}
 \frac{\Delta^lU\,\Delta^m\delta^jV\,\delta^{m+l-j}W}{m!j!(l-j)!}.
\end{align*}
The change $l=i+a$, $m=b$, $j=i$ identifies the sums. At exponent $q$, if the three input bounds are $R,S,T$, then $i+a+b\leq R+S+T-q$; for fixed derivative indices the three source exponents have fixed sum and individual upper bounds, hence individual lower bounds. All rearrangements are therefore finite.

Both $\Delta$ and $\delta$ annihilate $1$ and ${\red{\Bbbk}}$, giving the unit and central copy of ${\red{\Bbbk}}$. Commutation with $\Delta,\delta$ shows that $\delta$ and $\eta$ differentiate each finite coefficient sum of $\star$. Since $\Delta\DD=1$, $\Delta^2\DD=0$ and $\delta\DD=0$, we have $\DD\star U=\DD U+\delta U$ and $U\star\DD=U\DD$. Finally every coefficient formula is a finite expression in field operations and the two derivations. It is preserved by the stated embeddings; injectivity preserves nonzero coefficients and therefore orders.
\end{proof}

The coefficient field ${\red{\mathbb{F}}}$ is used as an ordinary coefficient ring; it is not asserted to be a commutative subalgebra for $\star$.

\begin{lemma}\label{lem:twist}
For $c\in\mathbb Q$ one can extend ${\red{\mathbb{F}}}$ to a commuting differential field containing a nonzero $e$ with $\delta e=0$, $\eta e=ce$. On a lattice containing $c$, the symbol $e\DD^{-c}$ is central and invertible, and multiplication by it for $\star$ is ordinary multiplication.
\end{lemma}
\begin{proof}
Take $e$ transcendental and prescribe the displayed derivatives on ${\red{\mathbb{F}}}(e)$. Their commutator is a derivation vanishing on ${\red{\mathbb{F}}}$ and $e$, hence on the field. For $C=e\DD^{-c}$ and $C^{-1}=e^{-1}\DD^c$ one has
$\Delta C=\delta C=\Delta C^{-1}=\delta C^{-1}=0$. All positive terms of~\eqref{eq:star} in a product with these elements vanish on either side. Their ordinary product is $1$.
\end{proof}

\begin{lemma}\label{lem:separation}
For every nonzero $f\in {\red{\Bbbk}}\langle x,y\rangle$, the generic symbol $f(\DD,v)$ is nonzero. Moreover, if every word of $f$ contains $x$, contains exactly $k_0\geq1$ copies of $y$, and contains at most $A_0$ copies of $x$, then
\[
 \ord f(\DD,v)\geq1-A_0k_0.
\]
\end{lemma}
\begin{proof}
Take an upper bound $A_0$ on the $x$-counts, independent variables $\lambda_0,\ldots,\lambda_{A_0}$, and ${\red{\mathbb{K}}}={\red{\Bbbk}}(\lambda_0,\ldots,\lambda_{A_0})$. On ${\red{\mathbb{K}}}[\tau]$ let $D_0=d/d\tau$, $E=\tau D_0$, and
\[
 P(T)=\sum_{j=0}^{A_0}\lambda_j
   \prod_{\substack{0\leq i\leq A_0\\i\ne j}}\frac{T-i}{j-i},
 \qquad B=P(E).
\]
The denominators are nonzero in characteristic zero. On $\epsilon_j=\tau^j/j!$, for $0\leq j\leq A_0$, we have $D_0\epsilon_j=\epsilon_{j-1}$ (with $D_0\epsilon_0=0$) and $B\epsilon_j=\lambda_j\epsilon_j$. A word with $a$ occurrences of $x$ has the unique form
$w=y^{b_0}xy^{b_1}\cdots xy^{b_a}$, and
\[
 w(D_0,B)\epsilon_{A_0}
 =\left(\prod_{j=0}^a\lambda_{A_0-a+j}^{b_j}\right)\epsilon_{A_0-a}.
 \tag{3.2}\label{eq:separateword}
\]
For fixed $a$, distinct words give distinct monomials in the independent $\lambda$'s. Different $a$ give different basis vectors. Thus $f(D_0,B)\ne0$.

The monic falling factorials $(T)_j=T(T-1)\cdots(T-j+1)$, $0\leq j\leq A_0$, form a basis of the polynomials of degree at most $A_0$, by successive cancellation of leading terms. Write $P(T)=\sum c_j(T)_j$. The relation $D_0\tau=\tau D_0+1$ gives
$(E)_j=\tau^jD_0^j$: induction uses $D_0^j\tau=\tau D_0^j+jD_0^{j-1}$ and multiplies $(E)_j$ by $E-j$. Hence
\[
 B=\sum_{j=0}^{A_0}c_j\tau^jD_0^j.
\]
On ${\red{\mathbb{K}}}(\tau,u)$ put $\delta=\partial_\tau$ and $\eta=u\partial_u$, both killing ${\red{\mathbb{K}}}$, and set $b=\sum c_j\tau^ju^j$. For finite sums of monomials $\tau^\ell u^q\DD^r$, where $\ell,q\geq0$ and $r\in\mathbb Z$, formula~\eqref{eq:star} reads
\[
 (\tau^\ell u^q\DD^r)\star(\tau^m u^n\DD^s)
 =\sum_{i=0}^m\frac{(r+q)_i(m)_i}{i!}
 \tau^{\ell+m-i}u^{q+n}\DD^{r+s-i}.
 \tag{3.3}\label{eq:specialize}
\]
This space is closed under $\star$. The map
$\Phi(\tau^\ell u^q\DD^r)=\tau^\ell D_0^{r+q}$ to symbols with $\eta=0$, $\delta=\partial_\tau$, preserves products by~\eqref{eq:specialize}. It sends $\DD$ to $D_0$ and $b$ to $B$. The ordinary differential symbols of nonnegative $D_0$-degree act on ${\red{\mathbb{K}}}[\tau]$ by the Leibniz rule, compatibly with multiplication. Therefore $\Phi(f(\DD,b))=f(D_0,B)\ne0$, so $f(\DD,b)\ne0$.

Each coefficient of $f(\DD,v)$ is a polynomial in finitely many jets: this follows by induction on word length from the finite coefficient calculation in Lemma~\ref{lem:algebra}, without inversion of jets. Sending $v_{ij}$ to $\delta^i\eta^jb$ is a differential polynomial-ring homomorphism, and specializes the symbol to $f(\DD,b)$. Thus the generic symbol is nonzero.

Under the additional assumptions, the finite expression $f(\DD,b)$ has the form
\[
 \sum_r\sum_{q=0}^{A_0k_0} a_{r,q}(\tau)u^q\DD^r,
\]
because there are $k_0$ copies of $b$, each of $u$-degree at most $A_0$, and the product formula never raises their summed $u$-degree. The nonzero operator $f(D_0,B)$ annihilates $1$: in each word its rightmost $D_0$ has only powers of $B$ to its right, which send $1$ to a scalar, then killed by $D_0$. A polynomial differential operator of order at most zero is multiplication by a function; if it annihilates $1$ that function is zero. Hence $\ord f(D_0,B)\geq1$. If all occupied integer exponents $r$ above were less than $1-A_0k_0$, then $r+q\leq0$ for all terms, contradicting that inequality after applying $\Phi$. Some occupied $r$ is therefore at least $1-A_0k_0$. Specialization cannot turn an identically zero generic coefficient into a nonzero one, proving the required lower bound.
\end{proof}

For a central differential constant $t$ define the variations by
\[
 h(\DD,Z+tV)=\sum_{k=0}^M t^k B_k(Z;V),\qquad B_0(Z;V)=h(\DD,Z).
 \tag{3.4}\label{eq:variations}
\]
Here and throughout, free-polynomial evaluation in symbols uses $\star$.

\begin{lemma}\label{lem:variation}
For every $Z\in\Sf_p({\red{\mathbb{F}}})$ there is a nonempty $I\subseteq\{1,\ldots,M\}$ and numbers $r_k\in L_p$, $k\in I$, such that for every rational $s$,
\[
 \ord B_k(Z;v\DD^{-s})=-r_k-ks\quad(k\in I).
 \tag{3.5}\label{eq:covariance}
\]
Each leading coefficient is a nonzero polynomial in finitely many jets, homogeneous of ordinary degree $k$. For $k\notin I$, $B_k(Z;V)=0$ for every symbol $V$, on any common enlarged lattice and differential extension. At $Z=0$ with initial field ${\red{\Bbbk}}$ and zero derivations,
\[
 s_0:=\max_{k\in I}\frac{-r_k}{k}\geq-A+\frac1M.
 \tag{3.6}\label{eq:initialslope}
\]
\end{lemma}
\begin{proof}
In a word expansion, $B_k$ chooses exactly $k$ $y$-positions for $V$. Its coefficients at $V=v\DD^{-s}$ are finite jet polynomials homogeneous of ordinary degree $k$, with exponents in $-ks+L_p$ and an upper bound. Derivations preserve jet degree and lower the symbol exponent only by integers. Also, for a central symbol $C$,
\[
 B_k(Z;C\star V)=C^k\star B_k(Z;V),
 \tag{3.7}\label{eq:centralvariation}
\]
by moving the $k$ central factors out of each word.

At slope zero let $I$ consist of the nonzero $B_k(Z;v)$ and set $-r_k=\ord B_k(Z;v)$. For a rational $s$, adjoin $e$ with $\delta e=0$, $\eta e=se$, and let $w$ be generic over ${\red{\mathbb{F}}}(e)$. Substitution $v=ew$ sends
\[
 v_{ij}\longmapsto e\sum_{a=0}^j\binom ja s^{j-a}w_{ia}.
 \tag{3.8}\label{eq:jetsubstitution}
\]
In any finite set of jets closed under lowering $j$, this is an invertible triangular linear substitution, with diagonal coefficient $e\ne0$; the inverse is the binomial transform with $-s$ and multiplication by $e^{-1}$. It preserves nonvanishing of every coefficient polynomial. Lemma~\ref{lem:twist} and~\eqref{eq:centralvariation} yield
\[
 B_k(Z;(ew)\DD^{-s})=e^k\DD^{-ks}B_k(Z;w).
\]
Thus the same $I$ works at every slope, and the orders are exactly~\eqref{eq:covariance}.

We justify also the assertion for arbitrary $V$ when $k\notin I$. Polarize:
\[
 H_k(V_1,\ldots,V_k)=[t_1\cdots t_k]
 B_k\left(Z;\sum_i t_iV_i\right).
\]
This is symmetric, additive in each argument, and multilinear over central scalars; moreover $H_k(V,\ldots,V)=k!B_k(Z;V)$. The identity $B_k(Z;v)=0$ is coefficientwise a polynomial identity in the jets. Substituting $v=\sum t_i a_i$ for coefficient elements in any differential extension and comparing $t_1\cdots t_k$ gives $H_k(a_1,\ldots,a_k)=0$. For monomials $V_i=a_i\DD^{r_i}$, adjoin nonzero $e_i$ with $\delta e_i=0$, $\eta e_i=-r_ie_i$. Then $C_i=e_i\DD^{r_i}$ is central and
\[
 H_k(V_1,\ldots,V_k)
 =\left(\prod_i C_i\right)\star H_k(e_1^{-1}a_1,\ldots,e_k^{-1}a_k)=0.
\]
Injectivity of the extension gives the identity in the original field as well. Additivity proves it for finite Laurent sums. For a fixed coefficient of an arbitrary upper-bounded symbol $V$, each participating exponent has a lower bound obtained by subtracting the upper bounds of the other factors from the specified output exponent; differentiations only decrease the sum. Truncating $V$ below a common such bound leaves this coefficient unchanged and reduces to finite Laurent sums. Hence the identity holds for all $V$.

Finally $B_M(Z;V)=h_M(\DD,V)$, independently of $Z$. Lemma~\ref{lem:separation} makes this nonzero at generic $v$ over ${\red{\Bbbk}}$; its coefficient polynomials remain nonzero over ${\red{\mathbb{F}}}$ since the jets remain independent. Thus $M\in I$. At $Z=0$, $B_k(0;V)=h_k(\DD,V)$, so $I$ consists exactly of the nonzero $h_k$. Every word of such an $h_k$ contains $x$, and Lemma~\ref{lem:separation} gives
$-r_k\geq1-Ak$. Thus $-r_k/k\geq-A+1/k\geq-A+1/M$, proving~\eqref{eq:initialslope}.
\end{proof}

\section{Realizing a differential equation with inequations}\label{sec:branch}
We first record the elementary algebra used in this section. A polynomial ring in finitely many variables over a field is a unique factorization domain. One derivation is as follows. The one-variable division algorithm gives factorization by degree induction and primality of irreducibles through the Bezout identity. For a unique factorization domain ${\red{\mathcal U}}$, call a polynomial primitive when no irreducible of ${\red{\mathcal U}}$ divides all its coefficients. Reduction modulo each irreducible prime shows that the product of primitive polynomials is primitive. Splitting off coefficient contents, factoring over $\Frac({\red{\mathcal U}})$, clearing denominators and then removing contents yields factorization in ${\red{\mathcal U}}[T]$. Comparison of the minimum exponent of each prime among coefficients shows that two primitive associates over $\Frac({\red{\mathcal U}})$ differ by a unit of ${\red{\mathcal U}}$; this proves uniqueness. Induction on the number of variables gives the assertion. In particular an irreducible polynomial generates a prime ideal. Polynomial extension of its quotient is still a domain, including extension by an arbitrary set of independent variables, since each product involves only finitely many variables.

\begin{lemma}\label{lem:branch}
Let ${\red{\mathbb{F}}}/{\red{\Bbbk}}$ carry commuting derivations $\delta,\eta$ fixing ${\red{\Bbbk}}$, where ${\red{\Bbbk}}$ is algebraically closed of characteristic zero. Put
\[
 R_N={\red{\mathbb{F}}}[X_{ij}:i,j\geq0,\ i+j\leq N].
\]
Let $P\in R_N$ be nonconstant and irreducible, of total differential order exactly $N$. For $N=0$ this means $P\in {\red{\mathbb{F}}}[X_{00}]$. Let $Q_1,\ldots,Q_\ell\in R_N$ be nonzero and not divisible by $P$. Then there is a field extension ${\red{\mathbb{E}}}/{\red{\mathbb{F}}}$ carrying commuting extensions of $\delta,\eta$, and $v\in {\red{\mathbb{E}}}$, such that
\[
 P((\delta^i\eta^jv))=0,\qquad Q_r((\delta^i\eta^jv))\ne0\quad(1\leq r\leq\ell).
\]
The formal linear differential polynomial
\[
 {\red{\mathrm{d}}}P_v(w)=\sum_{i+j\leq N}P_{X_{ij}}((\delta^a\eta^bv))\,\delta^i\eta^jw
\]
in a new differential indeterminate $w$ is nonzero.
\end{lemma}
\begin{proof}
We use an explicit extension rule for derivations. If ${\red{\mathbb{L}}}={\red{\mathbb{K}}}(u)$ is a finite simple extension and its minimal polynomial ${\red{\mu}}$ satisfies ${\red{\mu}}'(u)\ne0$, every derivation ${\red{\partial}}:{\red{\mathbb{K}}}\to {\red{\mathbb{L}}}$ extends uniquely to ${\red{\mathbb{L}}}$, with
\[
 {\red{\partial}}(u)=-\frac{{\red{\mu}}^{\red{\partial}}(u)}{{\red{\mu}}'(u)},\qquad
 {\red{\mu}}^{\red{\partial}}({\red{\xi}})=\sum_m {\red{\partial}}(a_m){\red{\xi}}^m\quad\text{for }{\red{\mu}}({\red{\xi}})=\sum_m a_m{\red{\xi}}^m.
 \tag{4.1}\label{eq:derivationextension}
\]
This allows ${\red{\partial}}({\red{\mathbb{K}}})$ to lie in ${\red{\mathbb{L}}}$ rather than ${\red{\mathbb{K}}}$. Indeed, define on ${\red{\mathbb{K}}}[{\red{\xi}}]$ the sum of coefficient differentiation evaluated at $u$ and formal ${\red{\xi}}$-differentiation evaluated at $u$ times the displayed value of ${\red{\partial}}(u)$. The product rule holds and ${\red{\mu}}$ is killed, so this descends to ${\red{\mathbb{K}}}[{\red{\xi}}]/({\red{\mu}})={\red{\mathbb{K}}}(u)$. Differentiating ${\red{\mu}}(u)=0$ proves uniqueness. The commutator of two derivations is a derivation, by expanding a product and canceling mixed terms. If that commutator vanishes on ${\red{\mathbb{K}}}$, applying it to ${\red{\mu}}(u)$ shows it vanishes on $u$, since ${\red{\mu}}'(u)\ne0$.

For $N=0$, take ${\red{\mathbb{E}}}={\red{\mathbb{F}}}[X_{00}]/(P)$ and $v$ the class of $X_{00}$. Characteristic zero gives $P'\ne0$ of smaller degree than $P$, hence $P'(v)\ne0$. Extend both derivations by~\eqref{eq:derivationextension}; their commutator vanishes by the preceding argument. The evaluation kernel is exactly $(P)$, so every $Q_r(v)$ is nonzero and the linearization is $P'(v)w\ne0$.

Now suppose $N\geq1$. At least one highest-order partial derivative $P_{X_{N-j,j}}$ is nonzero. Thus
$G(C)=\sum_{j=0}^N C^jP_{X_{N-j,j}}$ is a nonzero polynomial over $\Frac(R_N)$. Its root set has at most $N$ elements, and the infinite field ${\red{\Bbbk}}$ contains a $c$ with $G(c)\ne0$. Set ${\red{\partial}}_{\red{\mathbb{F}}}=\delta$ and ${\red{\nabla}}_{\red{\mathbb{F}}}=\eta-c\delta$. They commute and kill ${\red{\Bbbk}}$.

Use variables $T_{ab}$ for $a+b\leq N$ and the isomorphism
\[
 \Phi(X_{ij})=\sum_{b=0}^j\binom jb c^{j-b}T_{i+j-b,b}.
 \tag{4.2}\label{eq:directionchange}
\]
Its inverse sends $T_{ab}$ to $\sum_{j=0}^b\binom bj(-c)^{b-j}X_{a+b-j,j}$. These are inverse binomial substitutions for the commuting directions $\delta={\red{\partial}}$, $\eta={\red{\nabla}}+c{\red{\partial}}$. Put $\widetilde P=\Phi(P)$, $\widetilde Q_r=\Phi(Q_r)$ and $z=T_{N0}$. The chain rule gives
\[
 \widetilde P_z=\Phi(G(c))\ne0.
\]
Thus $\widetilde P$ has positive $z$-degree and remains irreducible.

Let
\[
 {\red{\mathcal B}}_0={\red{\mathbb{F}}}[T_{ab}:a+b\leq N,\ (a,b)\ne(N,0)],\qquad
 {\red{\mathcal B}}={\red{\mathbb{F}}}[T_{ab}:0\leq a<N,\ b\geq0].
\]
The domain
\[
 {\red{\mathcal C}}={\red{\mathcal B}}[z]/(\widetilde P)
 =\bigl({\red{\mathcal B}}_0[z]/(\widetilde P)\bigr)
 [T_{ab}:0\leq a<N,\ a+b>N]
\]
contains ${\red{\mathcal B}}$ injectively: a nonzero multiple of $\widetilde P$ in ${\red{\mathcal B}}[z]$ has positive $z$-degree and cannot belong to ${\red{\mathcal B}}\setminus\{0\}$. Let ${\red{\mathbb{K}}}=\Frac({\red{\mathcal B}})$, ${\red{\mathbb{E}}}=\Frac({\red{\mathcal C}})$, let $t_{ab}$ denote the variables in ${\red{\mathbb{K}}}$, and let $u$ denote the class of $z$. Then ${\red{\mathbb{E}}}={\red{\mathbb{K}}}(u)$, and $u$ is algebraic over ${\red{\mathbb{K}}}$. The localization ${\red{\mathbb{K}}}[z]/(\widetilde P)$ is a finite-dimensional domain over ${\red{\mathbb{K}}}$, hence a field: multiplication by a nonzero element is injective and therefore surjective. Thus a scalar multiple of $\widetilde P$ is the minimal polynomial of $u$ over ${\red{\mathbb{K}}}$.

Evaluation ${\red{\operatorname{ev}}}:{\red{\mathcal B}}_0[z]\to {\red{\mathbb{E}}}$ has kernel exactly $(\widetilde P)$, because the quotient injects first into its polynomial extension ${\red{\mathcal C}}$ and then into its fraction field. Hence
\[
 {\red{\varsigma}}={\red{\operatorname{ev}}}(\widetilde P_z)\ne0,\qquad {\red{\operatorname{ev}}}(\widetilde Q_r)\ne0.
 \tag{4.3}\label{eq:separant}
\]
For the first assertion, $\widetilde P_z$ is nonzero and has smaller $z$-degree than $\widetilde P$; for the others, use the stated nondivisibility, preserved by $\Phi$.

Define ${\red{\nabla}}_0:{\red{\mathbb{K}}}\to {\red{\mathbb{K}}}$ by ${\red{\nabla}}_0|_{\red{\mathbb{F}}}=\eta-c\delta$ and ${\red{\nabla}}_0(t_{ab})=t_{a,b+1}$. Independence of the variables defines it on ${\red{\mathcal B}}$, and the quotient rule defines it on ${\red{\mathbb{K}}}$. Extend it to ${\red{\nabla}}:{\red{\mathbb{E}}}\to {\red{\mathbb{E}}}$ by~\eqref{eq:derivationextension}. Next define ${\red{\partial}}_0:{\red{\mathbb{K}}}\to {\red{\mathbb{E}}}$ by ${\red{\partial}}_0|_{\red{\mathbb{F}}}=\delta$ and
\[
 {\red{\partial}}_0(t_{ab})=
 \begin{cases}t_{a+1,b},&a<N-1,\\ {\red{\nabla}}^b(u),&a=N-1.\end{cases}
 \tag{4.4}\label{eq:constructD}
\]
Again independence and the quotient rule define this derivation. Its values are permitted to lie in ${\red{\mathbb{E}}}$. Extend it to ${\red{\partial}}:{\red{\mathbb{E}}}\to {\red{\mathbb{E}}}$ by~\eqref{eq:derivationextension}.

We check commutativity after constructing both derivations. On ${\red{\mathbb{F}}}$ it holds. For $a<N-1$ both ${\red{\partial}}{\red{\nabla}}(t_{ab})$ and ${\red{\nabla}}{\red{\partial}}(t_{ab})$ equal $t_{a+1,b+1}$; for $a=N-1$ they both equal ${\red{\nabla}}^{b+1}(u)$. Thus $[{\red{\partial}},{\red{\nabla}}]$ vanishes on ${\red{\mathbb{K}}}$, and~\eqref{eq:separant} and the commutator argument following~\eqref{eq:derivationextension} make it vanish on ${\red{\mathbb{E}}}$.

Set $v=t_{00}$. Then ${\red{\partial}}^a{\red{\nabla}}^bv=t_{ab}$ for $a<N$, while ${\red{\partial}}^Nv=u$. These are every pair $a+b\leq N$. Define on ${\red{\mathbb{E}}}$ the extensions $\delta_{\red{\mathbb{E}}}={\red{\partial}}$, $\eta_{\red{\mathbb{E}}}={\red{\nabla}}+c{\red{\partial}}$. They restrict correctly to ${\red{\mathbb{F}}}$, commute, and kill ${\red{\Bbbk}}$. Their jets satisfy
\[
 \delta_{\red{\mathbb{E}}}^i\eta_{\red{\mathbb{E}}}^jv=\sum_{b=0}^j\binom jb c^{j-b}{\red{\partial}}^{i+j-b}{\red{\nabla}}^bv.
\]
Thus evaluation at these jets is ${\red{\operatorname{ev}}}\circ\Phi$, whose kernel on $R_N$ is exactly $(P)$. This proves the equation and all the inequations.

Finally evaluating $\widetilde P_z=\Phi(G(c))$ yields
\[
 0\ne {\red{\varsigma}}=\sum_{j=0}^N c^jP_{X_{N-j,j}}((\delta_{\red{\mathbb{E}}}^a\eta_{\red{\mathbb{E}}}^bv)).
\]
Some displayed partial-derivative value is nonzero, so the coefficient of the corresponding independent jet of $w$ in ${\red{\mathrm{d}}}P_v(w)$ is nonzero. The conclusion concerns this formal polynomial, not its action on every fixed differential field.
\end{proof}

\section{Finite Newton approximation}\label{sec:newton}
\begin{lemma}\label{lem:step}
Suppose $Z\in\Sf_p({\red{\mathbb{F}}})$ has nonzero residual with leading term $c\DD^{-\rho}$, where $c\ne0$ and $\rho\in L_p$. With the data of Lemma~\ref{lem:variation}, put
\[
 s=\max_{k\in I}\frac{\rho-r_k}{k},\qquad
 p'=\operatorname{lcm}(p,\operatorname{den}(s)).
\]
The coefficient $\Psi(v)$ at $\DD^{-\rho}$ in $h(\DD,Z+v\DD^{-s})$ is a nonconstant polynomial in finitely many jets, with $\Psi(0)=c$ and ordinary degree $q\leq M$. There are no coefficients above $-\rho$. Select an irreducible factor $P$ of maximal total differential order, with exact multiplicity $m$, and write $\Psi=P^mQ$, $P\nmid Q$. There is a commuting differential extension ${\red{\mathbb{E}}}/{\red{\mathbb{F}}}$ and $v\in {\red{\mathbb{E}}}$ such that $P(v)=0$, $Q(v)\ne0$, ${\red{\mathrm{d}}}P_v\ne0$. For $Z'=Z+v\DD^{-s}\in\Sf_{p'}({\red{\mathbb{E}}})$, the residual is zero or has leading exponent $-\rho'$ with $\rho'>\rho$. In the latter case the next slope $s'$ is strictly greater than $s$, and the next leading polynomial has degree at most $m$.
\end{lemma}
\begin{proof}
At every rational slope, all coefficients of $h(\DD,Z+u\DD^{-s})$ lie in ${\red{\mathbb{F}}}\{u\}$. Each is a finite polynomial, by the coefficient bounds in Lemma~\ref{lem:algebra}. Consequently the differential substitution $u\mapsto v+w$ into a field extension commutes with symbol evaluation; it never creates a coefficient where the universal polynomial was zero. No specialization of a differential fraction field is involved.

Formula~\eqref{eq:covariance} and the maximum defining $s$ show that all variations have order at most $-\rho$. With
$J=\{k\in I:r_k+ks=\rho\}$, which is nonempty, the leading coefficient is
\[
 \Psi(u)=c+\sum_{k\in J}C_k(u),
 \tag{5.1}\label{eq:leadingequation}
\]
where each $C_k$ is a nonzero homogeneous finite jet polynomial of ordinary degree $k$. Distinct degrees cannot cancel, so $\Psi$ is nonconstant and $q=\max J\leq M$.

Factor over the current field ${\red{\mathbb{F}}}$ in a finite jet ring, and select $P$ of greatest total differential order $N$. Every other factor, hence $Q$, belongs to $R_N={\red{\mathbb{F}}}[u_{ij}:i+j\leq N]$. The selected $P$ is irreducible in that ring: a factorization there would be a factorization in the original ring. Its exact multiplicity gives $P\nmid Q$, and $\Psi(0)=c\ne0$ implies $P(0),Q(0)\ne0$. Choose a variable on which $P$ depends. Its partial derivative $H$ is nonzero in characteristic zero, belongs to $R_N$, and has smaller ordinary degree than $P$, so $P\nmid H$. Lemma~\ref{lem:branch} applies to $P$ with the inequations $Q,H$. It yields $P(v)=0$, $Q(v)\ne0$, and a nonzero linearization
$L(w)={\red{\mathrm{d}}}P_v(w)$, since the coefficient given by $H(v)$ is nonzero.

The exact choice of $p'$ contains both the old lattice and $s$. The new symbol is therefore legal. Substitution $u\mapsto v$ kills its coefficient $\Psi(v)$ at $-\rho$ and leaves every higher coefficient zero. Thus the residual is zero or has strictly smaller order on the discrete lattice $L_{p'}$, giving $\rho'>\rho$.

For the quantitative assertion suppose this residual is nonzero. Expanding polynomials in their independent jet variables gives
\begin{align*}
 &P(v+w)=L(w)+\text{terms of degree at least two},\\
 &Q(v+w)=Q(v)+\text{terms of degree at least one}.
\end{align*}
Hence the lowest nonzero homogeneous degree of $\Psi(v+w)$ is exactly $m$, and its degree-$m$ term is
\[
 Q(v)L(w)^m\ne0.\tag{5.2}\label{eq:translateddegree}
\]
Nonvanishing uses the integral domain ${\red{\mathbb{E}}}\{w\}$.

Let $I',r'_k$ be the variation data at $Z'$. The identity
\[
 h(\DD,Z'+w\DD^{-s})=h(\DD,Z+(v+w)\DD^{-s})
\]
has no coefficients above $-\rho$, and its coefficient at $-\rho$ is $\Psi(v+w)$. The $k$th variation on the left is homogeneous of ordinary $w$-jet degree $k$, so uniqueness of homogeneous components shows individually that all its coefficients above $-\rho$ vanish. At $-\rho$ its coefficient is $[\Psi(v+w)]_k$. Formula~\eqref{eq:translateddegree} therefore gives, for $a_k=r'_k+ks$,
\[
 a_k\geq\rho\ (k\in I'),\qquad m\in I',\quad a_m=\rho,
 \qquad a_k>\rho\ (k\in I',\ k<m).
\]
Put $\epsilon=\rho'-\rho>0$ and
$\sigma_k=(\rho'-r'_k)/k=s+(\rho'-a_k)/k$. Then
\[
 \sigma_m=s+\frac\epsilon m>s,
 \qquad
 \sigma_k\leq s+\frac\epsilon k<s+\frac\epsilon m
 \quad(k>m).
\]
Thus $s'=\max_{k\in I'}\sigma_k>s$, and no index greater than $m$ attains this maximum. Applying~\eqref{eq:covariance} at $s'$ shows that the next leading polynomial consists of its nonzero constant term and the nonzero homogeneous components with $\sigma_k=s'$. Its degree is consequently at most $m$. This uses exact slope covariance at the updated $Z'$, not merely the translation identity at the old slope.
\end{proof}

\begin{lemma}\label{lem:orbit}
Let ${\red{\mathbb{F}}}$ be a characteristic-zero field containing all $e$th roots of unity, and let $\Psi\in {\red{\mathbb{F}}}[z_1,\ldots,z_N]$ be nonconstant of ordinary degree $q$, with $\Psi(0)\ne0$. Suppose $\Psi(\zeta z)=\Psi(z)$ whenever $\zeta^e=1$. If an irreducible factor $P$ has exact multiplicity $m$, then $q\geq em$.
\end{lemma}
\begin{proof}
Normalize $P(0)=1$, possible because every factor of $\Psi$ has nonzero constant term. This representative is unique within its associate class. The group ${\red{\mathcal G}}=\mu_e({\red{\mathbb{F}}})$ has exactly $e$ elements: $T^e-1$ splits and its derivative $eT^{e-1}$ is nonzero at every root. It acts by $\sigma_\zeta {\red{\Theta}}(z)={\red{\Theta}}(\zeta z)$. Let ${\red{o}}$ be the number of distinct orbit factors and ${\red{\mathcal H}}$ the stabilizer of $P$. The fibers of $\zeta\mapsto\sigma_\zeta P$ are precisely cosets of ${\red{\mathcal H}}$, so $e={\red{o}}|{\red{\mathcal H}}|$.

All orbit factors are irreducible, have the same degree $\ell=\deg P$, and are pairwise nonassociate because their constant terms are $1$. If $\Psi=P^mQ$ with $P\nmid Q$, invariance gives $\Psi=(\sigma_\zeta P)^m\sigma_\zeta Q$; applying the inverse substitution shows $\sigma_\zeta P\nmid\sigma_\zeta Q$. Hence every orbit factor has exact multiplicity $m$.

Choose a nonconstant monomial of $P$ of degree $j$, with $1\leq j\leq\ell$. Comparing its nonzero coefficient under $\sigma_\zeta P=P$ shows that every $\zeta\in {\red{\mathcal H}}$ satisfies $\zeta^j=1$. The root bound gives $|{\red{\mathcal H}}|\leq j\leq\ell$. Unique factorization implies that the product of the ${\red{o}}$ orbit factors, each raised to $m$, divides $\Psi$. Total degrees add in this polynomial domain, since the top homogeneous components have nonzero product. Therefore $q\geq {\red{o}}m\ell\geq {\red{o}}m|{\red{\mathcal H}}|=em$. All factorizations are over the stated field ${\red{\mathbb{F}}}$.
\end{proof}

For the Newton data in Lemma~\ref{lem:step}, put $e=p'/p$. If $k$ is active, so $ks=\rho-r_k\in L_p$, then
\[
 e\mid k.\tag{5.3}\label{eq:denominator}
\]
Indeed write $s=a/b$ in lowest terms, with $b=1$ if $s=0$. Then $e=b/\gcd(p,b)$. Integrality of $pka/b$ and $\gcd(a,b)=1$ imply $b\mid pk$ by Bezout's identity. Dividing $p,b$ by their greatest common divisor and using coprimality gives $e\mid k$. Thus every nonconstant homogeneous component in~\eqref{eq:leadingequation} has degree divisible by $e$, and the polynomial is invariant under simultaneous multiplication of its jets by an $e$th root of unity. These roots lie in ${\red{\Bbbk}}\subset {\red{\mathbb{F}}}$. Lemma~\ref{lem:orbit}, applied before any differential extension, gives $em\leq q$. This includes $e=1$ and $s=0$.

\begin{proposition}\label{prop:finite}
For every positive integer $T$ there are a commuting differential field ${\red{\mathbb{F}}}_T/{\red{\Bbbk}}$, an integer $1\leq p_T\leq M$, and a finite Laurent sum $Z_T\in\Sf_{p_T}({\red{\mathbb{F}}}_T)$ such that
\[
 \ord Z_T<A,\qquad \ord h(\DD,Z_T)\leq-T.
 \tag{5.4}\label{eq:precision}
\]
At most $MT$ corrections from Lemma~\ref{lem:step} are required.
\end{proposition}
\begin{proof}
Start with ${\red{\mathbb{F}}}_0={\red{\Bbbk}}$, zero derivations, $p_0=1$, $Z_0=0$, $\rho_0=0$, and $b_0=M$. Every nonconstant word of $h$ contains $y$, so $h(\DD,0)=1$. At a nonterminal stage $j$ maintain a finite symbol $Z_j$ in a compatible field and nested lattice, a nonzero residual with leading exponent $-\rho_j$, and a positive integer $b_j$ such that
\[
 p_jb_j\leq M,\qquad \rho_j\geq j/M,\qquad
 1\leq q_j:=\deg\Psi_j\leq b_j.
 \tag{5.5}\label{eq:invariant}
\]
Stop if $\rho_j\geq T$. The leading polynomial and slope exist by Lemmas~\ref{lem:variation} and~\ref{lem:step}. Initially $q_0\leq M$ and the first slope satisfies $s_0\geq-A+1/M$ by~\eqref{eq:initialslope}.

When another correction is needed, set $p_{j+1}=\operatorname{lcm}(p_j,\operatorname{den}(s_j))$, and select a maximal-differential-order irreducible factor of $\Psi_j$ with multiplicity $m_j$. Over the current field the root-of-unity argument above gives
\[
 p_{j+1}m_j=p_j(p_{j+1}/p_j)m_j\leq p_jq_j\leq p_jb_j\leq M.
 \tag{5.6}\label{eq:pm}
\]
This bound is obtained before extending the field or testing the next residual; it applies also when that residual will be zero.

Lemma~\ref{lem:step} provides a compatible commuting differential field extension and a correction $v_j\DD^{-s_j}$. Let $Z_{j+1}=Z_j+v_j\DD^{-s_j}$ and $b_{j+1}=m_j$. The old lattice is contained in the new one and the field embedding is injective, so old operations and orders are preserved by Lemma~\ref{lem:algebra}. The new symbol is a sum of $j+1$ monomials. If its residual is zero, stop. Otherwise Lemma~\ref{lem:step} gives $\rho_{j+1}>\rho_j$, $s_{j+1}>s_j$, and $q_{j+1}\leq m_j=b_{j+1}$. Both $\rho_j,\rho_{j+1}$ lie in $L_{p_{j+1}}$, so
\[
 \rho_{j+1}-\rho_j\geq\frac1{p_{j+1}}\geq\frac1M.
\]
Here $p_{j+1}\leq M$ follows from~\eqref{eq:pm}. This proves every invariant in~\eqref{eq:invariant} at the next nonterminal stage.

This is a finite induction: unless an exact zero or the required precision occurs sooner, after $MT$ corrections $\rho_{MT}\geq T$. Let ${\red{j_*}}\leq MT$ be the stopping stage, and take ${\red{\mathbb{F}}}_T={\red{\mathbb{F}}}_{\red{j_*}}$, $p_T=p_{\red{j_*}}$, $Z_T=Z_{\red{j_*}}$. The last use of~\eqref{eq:pm} ensures $p_T\leq M$ even for a zero residual. The correction slopes are strictly increasing, so every nonzero correction has order at most
$-s_0\leq A-1/M<A$. Their finite sum has no larger order. The residual has order at most $-T$ at stopping, with $\ord0=-\infty$. This proves~\eqref{eq:precision}. No infinite iteration or replacement of the nested lattices by $M^{-1}\mathbb Z$ is used.
\end{proof}

\section{Realization and finite matrix compression}\label{sec:matrix}
Throughout this section ${\red{\mathbb{F}}}/{\red{\Bbbk}}$ has commuting derivations fixing ${\red{\Bbbk}}$, and $p\geq1$. Define
\[
 {\red{\mathcal I}}_p=\{(a,c)\in\mathbb Z^2:a+c\geq0,\ a+c\equiv0\pmod p\}.
\]
Rows are output indices and columns input indices. For $r\in L_p$ let $\Af_p(r)$ consist of ${\red{\mathbb{F}}}$-valued matrices indexed by ${\red{\mathcal I}}_p$ whose nonzero entries from $(a,c)$ to $(a',c')$ satisfy
\[
 a'-a\leq pr,\qquad c'-c\geq-pr.
 \tag{6.1}\label{eq:jumps}
\]
Put $\Af_p=\bigcup_{r\in L_p}\Af_p(r)$.

\begin{lemma}\label{lem:realization}
Matrix multiplication in $\Af_p$ is well-defined, with every entry of every fixed finite product a finite sum. There is a unital ${\red{\Bbbk}}$-algebra homomorphism
\[
 \varrho:\Sf_p({\red{\mathbb{F}}})\longrightarrow\Af_p
\]
such that $\ord U\leq r$ implies $\varrho(U)\in\Af_p(r)$.
\end{lemma}
\begin{proof}
First consider a product $M_m\cdots M_1$, with $M_j\in\Af_p(r_j)$. A nonzero path with endpoints $\alpha_0=(a_0,c_0)$ and $\alpha_m=(a_m,c_m)$ has intermediate indices satisfying
\begin{align}
 a_m-p\sum_{{\red{\nu}}=j+1}^m r_{\red{\nu}}&\leq a_j\leq a_0+p\sum_{{\red{\nu}}=1}^j r_{\red{\nu}},\tag{6.2}\label{eq:patha}\\
 c_0-p\sum_{{\red{\nu}}=1}^j r_{\red{\nu}}&\leq c_j\leq c_m+p\sum_{{\red{\nu}}=j+1}^m r_{\red{\nu}}.\tag{6.3}\label{eq:pathc}
\end{align}
These follow by summing~\eqref{eq:jumps} along the first and last parts of the path. Each intermediate index lies in a finite integer rectangle, so there are only finitely many paths. Addition is closed after taking the larger bound, and summing~\eqref{eq:jumps} gives $\Af_p(r)\Af_p(s)\subseteq\Af_p(r+s)$. Associativity and distributivity regroup finite path sums. The identity belongs to $\Af_p(0)$; scalar matrices over ${\red{\mathbb{F}}}$ are central. Thus this is an ordinary unital matrix algebra with the stated finite-product property.

For the realization, first define the double Taylor map
\[
 {\red{\mathcal T}}(a)=\sum_{i,j\geq0}\frac{\delta^i\eta^j(a)}{i!j!}t^iw^j\in {\red{\mathbb{F}}}[[t,w]],
 \tag{6.4}\label{eq:taylor}
\]
where formal differentiation in $t,w$ holds the coefficient copy of ${\red{\mathbb{F}}}$ constant. The iterated Leibniz rule for commuting derivations gives
\[
 \delta^i\eta^j(ab)=\sum_{u=0}^i\sum_{v=0}^j
 \binom iu\binom jv\delta^u\eta^v(a)\delta^{i-u}\eta^{j-v}(b).
\]
Hence ${\red{\mathcal T}}$ is an injective unital algebra map (its constant coefficient is $a$), fixes ${\red{\Bbbk}}$, and satisfies ${\red{\mathcal T}}\delta=\partial_t{\red{\mathcal T}}$, ${\red{\mathcal T}}\eta=\partial_w{\red{\mathcal T}}$.

Extend ${\red{\mathcal T}}$ to symbol coefficients, writing $z$ for the symbol variable in these coordinates. Its values have the form
\[
 \widetilde U=\sum_{r\leq r_U}\sum_{i,j\geq0}u_{r,i,j}t^iw^jz^r,
 \qquad u_{r,i,j}=\frac{\delta^i\eta^j(u_r)}{i!j!}.
\]
Let $L_0=\partial_z+z^{-1}\partial_w$, with $\partial_z z^r=rz^{r-1}$. This commutes with $\partial_t$ and intertwines $\Delta$ under ${\red{\mathcal T}}$. On all expressions of the displayed upper-bounded type put
\[
 P\diamond Q=\sum_{n\geq0}\frac{L_0^nP\,\partial_t^nQ}{n!}.
 \tag{6.5}\label{eq:diamond}
\]
For fixed $z$-exponent $q$, a contribution uses $r+s-n=q$ with bounded-above $r,s$, so there are finitely many $n,r,s$. For each of these, any specified $t,w$ coefficient is a finite power-series calculation. Thus the expression is well-defined, and coefficient comparison gives
\[
 \widetilde{U\star V}=\widetilde U\diamond\widetilde V.
 \tag{6.6}\label{eq:intertwine}
\]

Use the Laurent polynomial space $E_0={\red{\mathbb{F}}}[z^{1/p},z^{-1/p},w]$ with basis $e_{a,b}=z^{a/p}w^b$, where $a\in\mathbb Z$, $b\geq0$. It is indexed by ${\red{\mathcal I}}_p$ on putting $c=pb-a$. On this space let $Z^r,W$ be multiplication by $z^r,w$ respectively, and set
\[
 T_0=-\partial_z-z^{-1}\partial_w.
\]
Define a normal-order map on monomials by
\[
 \mathcal Q(t^iw^jz^r)=T_0^iW^jZ^r.
 \tag{6.7}\label{eq:normalmap}
\]
The rightmost operator acts first. Each term of its application to $e_{a,b}$ is a scalar multiple of
\[
 z^{a/p+r-i}w^{b+j-\ell},\qquad
 0\leq\ell\leq i,\quad b+j-\ell\geq0.
 \tag{6.8}\label{eq:normalterms}
\]
Induction on $i$ proves this: the first summand of $T_0$ lowers just the $z$-exponent; the second lowers both exponents and kills a term with $w$-exponent zero rather than making it negative.

For a general upper-bounded expression $P$, define $\mathcal Q(P)$ entrywise by summing these matrices. Fix input $(a,b)$ and output $(a',b')$, and put $d_a=(a'-a)/p$. A contributing monomial must satisfy
\[
 r=d_a+i,\qquad j=b'-b+\ell,\qquad 0\leq\ell\leq i.
\]
The bound $r\leq r_P$ gives $0\leq i\leq r_P-d_a$; there are finitely many $i$, then finitely many $\ell$, and $r,j$ are determined. The entry is finite. Formula~\eqref{eq:normalterms} also gives
\[
 {\red{a'-a}}=p(r-i)\leq pr,\qquad
 {\red{c'-c}}=p(i+j-\ell-r)\geq-pr.
\]
Thus $\mathcal Q(P)\in\Af_p(r_P)$. The map is additive and ${\red{\mathbb{F}}}$-linear on Taylor-coordinate expressions, and sends $1$ to the identity.

We next check multiplication, including its extension to infinite expressions. If $B\in {\red{\mathbb{F}}}[w,z^{1/p},z^{-1/p}]$ and $M_B$ denotes multiplication by $B$, then the Leibniz rule gives
\[
 M_BT_0-T_0M_B=M_{L_0B}.
\]
Induction, multiplying on the right by $T_0$ and combining binomial coefficients, gives
\[
 M_BT_0^k=\sum_{n=0}^k\binom kn T_0^{k-n}M_{L_0^nB}.
\]
For $P=t^iB$ and $Q=t^kC$ one has
\[
 P\diamond Q=\sum_{n=0}^k\binom kn t^{i+k-n}(L_0^nB)C,
\]
so
\[
 \mathcal Q(P\diamond Q)
 =T_0^iM_BT_0^kM_C=\mathcal Q(P)\mathcal Q(Q).
 \tag{6.9}\label{eq:finiteordering}
\]
This proves the identity for finite sums.

To pass to full expressions, fix an entry and put $d_b=b'-b$. Two source monomials $t^iw^jz^r$ and $t^kw^lz^s$, with $r\leq r_P$, $s\leq r_Q$, have diamond-product terms of the form
\[
 \text{scalar}\cdot t^{i+k-n}w^{j+l-e}z^{r+s-n},
 \qquad 0\leq n\leq k,\quad0\leq e\leq\min(j,n).
\]
After normal ordering, a term reaching the fixed output satisfies
\[
 d_a=r+s-i-k,\qquad d_b=j+l-e-{\red{\beta}},
 \qquad0\leq {\red{\beta}}\leq i+k-n.
\]
In particular $e+{\red{\beta}}\leq i+k$, and
\[
 0\leq i+k\leq r_P+r_Q-d_a,\qquad
 0\leq j+l\leq d_b+i+k.
\]
The first bounds $i,k$; it then fixes $r+s$, so the two upper bounds give finite ranges for $r,s$ in $L_p$. The second bounds $j,l$. The remaining indices $n,e,{\red{\beta}}$ have finite ranges. Hence expansion of this entry of $\mathcal Q(P\diamond Q)$ uses finitely many source terms. For the product $\mathcal Q(P)\mathcal Q(Q)$, equations~\eqref{eq:patha}--\eqref{eq:pathc} give finitely many intermediate indices, and the entrywise argument after~\eqref{eq:normalterms} gives finitely many source terms for each entry. Both sides can therefore sum~\eqref{eq:finiteordering} using only finite rearrangements.

It follows from~\eqref{eq:intertwine} that $\varrho(U)=\mathcal Q(\widetilde U)$ preserves products. It is additive, unital, and fixes ${\red{\Bbbk}}$, because ${\red{\mathcal T}}$ does. The support assertion has already been proved. General images are regarded as elements of the explicitly defined algebra $\Af_p$, not as endomorphisms of the direct-sum space $E_0$; only the finite monomial operators were used on that space.
\end{proof}

\begin{proposition}\label{prop:compression}
Let $H\in {\red{\Bbbk}}\langle x,y\rangle$ have degree $D\geq1$, and suppose $U,V\in\Sf_p({\red{\mathbb{F}}})$ satisfy $\ord U,\ord V\leq R$ for an integer $R\geq1$ and $\ord H(U,V)\leq-T$ for an integer $T\geq1$. For every integer $1\leq S\leq T$ there exist $X_S,Y_S\in M_{pS^2}({\red{\mathbb{F}}})$ such that
\[
 \rank_{\red{\mathbb{F}}} H(X_S,Y_S)\leq4DRpS.
 \tag{6.10}\label{eq:rankcompression}
\]
\end{proposition}
\begin{proof}
Put $N=pS$ and
\[
 J=\{(a,c):0\leq a,c<N,\ a+c\equiv0\pmod p\}\subset {\red{\mathcal I}}_p.
\]
For each of the $N$ choices of $a$, exactly $S$ choices of $c$ have the required residue, so $|J|=pS^2$. Write $M_J$ for the submatrix on $J$, and take $X_S=\varrho(U)_J$, $Y_S=\varrho(V)_J$.

The residual image has nonzero entries only if $a'-a\leq-pT$, by Lemma~\ref{lem:realization}. Within $J$ we instead have
$a'-a\geq-(pS-1)>-pS\geq-pT$. Therefore
\[
 \varrho(H(U,V))_J=0.\tag{6.11}\label{eq:residualzero}
\]

Set $B=DRp$ and define common input and output boundary sets
\begin{align*}
 J_{\mathrm{in}}&=\{(a,c)\in J:a\geq N-B\ \text{or}\ c<B\},\\
 J_{\mathrm{out}}&=\{(a,c)\in J:a<B\ \text{or}\ c\geq N-B\}.
\end{align*}
For a word of length $1\leq m\leq D$, write its matrix factors in action order as $M_m\cdots M_1$, each a copy of $\varrho(U)$ or $\varrho(V)$. The entry of the full product is the finite sum over all paths in ${\red{\mathcal I}}_p$, while the product of the compressed factors uses exactly those paths staying in $J$. If its input is outside $J_{\mathrm{in}}$ and output outside $J_{\mathrm{out}}$, then
\[
 a_0<N-B,\quad c_0\geq B,\quad a_m\geq B,\quad c_m<N-B.
\]
The jump bounds and $m\leq D$ give for every intermediate index
\begin{align*}
 a_j&\leq a_0+jRp\leq a_0+B<N,&
 a_j&\geq a_m-(m-j)Rp\geq a_m-B\geq0,\\
 c_j&\geq c_0-jRp\geq c_0-B\geq0,&
 c_j&\leq c_m+(m-j)Rp\leq c_m+B<N.
\end{align*}
It already has the congruence defining ${\red{\mathcal I}}_p$, so it lies in $J$. Hence the difference between the compressed product and the compression of the full product is supported in rows $J_{\mathrm{out}}$ or columns $J_{\mathrm{in}}$. For the empty word the difference is zero, since the identity compresses to the identity. These same sets work for every word. Summing the words of $H$, using the homomorphism and~\eqref{eq:residualzero}, shows that $H(X_S,Y_S)$ has this common support.

For fixed $a$ or fixed $c$, there are exactly $S$ complementary coordinates in $J$. Each of the four strips contains at most $B$ possible restricted coordinates, even if $B\geq N$, and hence at most $BS$ indices. Thus
$|J_{\mathrm{in}}|,|J_{\mathrm{out}}|\leq2BS$.
For a matrix $Q$ supported in those rows or columns, the coordinate projections $P_{\mathrm{out}},P_{\mathrm{in}}$ give
\[
 Q=P_{\mathrm{out}}Q+(I-P_{\mathrm{out}})QP_{\mathrm{in}}.
\]
The two summands have rank at most $|J_{\mathrm{out}}|$ and $|J_{\mathrm{in}}|$, respectively. Since the image of a sum lies in the sum of its images, $\rank Q\leq4BS=4DRpS$. Apply this to $H(X_S,Y_S)$. The count handles overlapping or full strips and does not multiply the bound by the number of words of $H$.
\end{proof}

\section{Descent and completion of the proof}\label{sec:final}
\begin{lemma}\label{lem:descent}
Let ${\red{\Bbbk}}$ be algebraically closed, ${\red{\mathbb{L}}}/{\red{\Bbbk}}$ any field extension, $f\in {\red{\Bbbk}}\langle x_1,\ldots,x_d\rangle$, $n\geq1$, and $0\leq r\leq n$ an integer. If a tuple in $M_n({\red{\mathbb{L}}})^d$ gives a value of rank at most $r$, then some tuple in $M_n({\red{\Bbbk}})^d$ does so.
\end{lemma}
\begin{proof}
We first derive the finite-variable ideal fact needed here: every proper ideal in ${\red{\Bbbk}}[t_1,\ldots,t_N]$ has a common zero in ${\red{\Bbbk}}^N$. An algebraically closed field is infinite, since otherwise $\prod_{c\in {\red{\Bbbk}}}(T-c)-1$ would be a nonconstant polynomial without a root.

We claim that a field ${\red{\mathbb{F}}}$ finitely generated as a ${\red{\Bbbk}}$-algebra equals ${\red{\Bbbk}}$. Write ${\red{\mathbb{F}}}={\red{\Bbbk}}[a_1,\ldots,a_m]$ and choose a maximal algebraically independent subcollection $u_1,\ldots,u_s$ of these generators. Every remaining generator $b_j$ is algebraic over ${\red{\mathbb{K}}}={\red{\Bbbk}}(u_1,\ldots,u_s)$, since transcendence over ${\red{\mathbb{K}}}$ would enlarge that independent subcollection. Choose monic polynomial equations for the finitely many $b_j$, and a common nonzero polynomial denominator $q\in {\red{\Bbbk}}[u_1,\ldots,u_s]$ for their coefficients. Put ${\red{\mathcal R}}={\red{\Bbbk}}[u_1,\ldots,u_s,q^{-1}]$. Reduction of powers by these monic equations shows that ${\red{\mathcal R}}[b_1,\ldots,b_e]$ has a finite ${\red{\mathcal R}}$-module spanning list, consisting of products of bounded powers of the $b_j$ and including $1$. Moreover
\[
 {\red{\mathbb{F}}}={\red{\Bbbk}}[u_1,\ldots,u_s,b_1,\ldots,b_e]
 ={\red{\mathcal R}}[b_1,\ldots,b_e],
\]
because ${\red{\mathcal R}}\subset {\red{\mathbb{F}}}$.

If $v_1,\ldots,v_{\red{n_0}}$ is such a spanning list and $z\in {\red{\mathbb{F}}}$, write $zv_i=\sum_jc_{ij}v_j$ with $c_{ij}\in {\red{\mathcal R}}$. Multiplication by the adjugate of $zI-C$ shows that $\det(zI-C)v_i=0$ for all $i$. One $v_i$ is $1$, so this gives a monic equation for $z$ over ${\red{\mathcal R}}$. Apply it to $z=a^{-1}$ for a nonzero $a\in {\red{\mathcal R}}$, and multiply an equation of degree ${\red{n_0}}$ by $a^{{\red{n_0}}-1}$ to express $a^{-1}$ as an element of ${\red{\mathcal R}}$. Thus ${\red{\mathcal R}}$ is a field.

If $s\geq1$, view $q$ as a polynomial in $u_2,\ldots,u_s$ with coefficients in ${\red{\Bbbk}}[u_1]$. A nonzero such coefficient has finitely many roots, so there is $c\in {\red{\Bbbk}}$ such that $q(c,u_2,\ldots,u_s)\ne0$. The element $u_1-c$ cannot be invertible in ${\red{\mathcal R}}$: an inverse would be ${\red{\pi}}/q^\ell$ for a polynomial ${\red{\pi}}$ and $\ell\geq0$, giving $(u_1-c){\red{\pi}}=q^\ell$. Substitution $u_1=c$ gives zero on the left and a nonzero polynomial (or $1$ if $\ell=0$) on the right. This contradicts the field property. Therefore $s=0$. Every $a_i$ is algebraic over the algebraically closed ${\red{\Bbbk}}$, and its polynomial splits into linear factors; since ${\red{\mathbb{F}}}$ is a field, one factor vanishes. Thus each $a_i\in {\red{\Bbbk}}$ and ${\red{\mathbb{F}}}={\red{\Bbbk}}$.

For a proper ideal ${\red{\mathfrak a}}$ in ${\red{\mathcal P}}={\red{\Bbbk}}[t_1,\ldots,t_N]$, choose a maximal ideal $\mathfrak m$ containing it. Such an ideal exists by the maximal principle: a union of a chain of proper ideals containing ${\red{\mathfrak a}}$ is still proper, because otherwise one member would contain $1$. The field ${\red{\mathcal P}}/\mathfrak m$ contains ${\red{\Bbbk}}$ injectively and is generated over it by the finitely many coordinate images. The preceding assertion identifies it with ${\red{\Bbbk}}$. Those images form a tuple $a\in {\red{\Bbbk}}^N$ at which every member of ${\red{\mathfrak a}}$ vanishes. For $N=0$ the statement says that the proper ideal of ${\red{\Bbbk}}$ is zero and the tuple is empty.

We also record the matrix criterion: for a matrix $B$ over any field and an integer $0\leq r<n$, rank at most $r$ is equivalent to vanishing of all $(r+1)$-minors. A nonzero such minor gives $r+1$ independent columns. Conversely, choose $r+1$ independent columns if rank exceeds $r$. Their rows span the dual of the $(r+1)$-dimensional column-coordinate space: otherwise a nonzero vector annihilating their proper row span would give a nonzero kernel vector for those columns. Choose $r+1$ rows forming a basis of that dual to obtain a nonzero square minor. This proves the criterion, including $r=0$.

If $r=n$, any tuple over ${\red{\Bbbk}}$, for example the zero tuple, satisfies the desired rank bound. Suppose $r<n$ and form generic matrices $T_i=(t_{i,ab})$ over the finite polynomial ring ${\red{\mathcal P}}={\red{\Bbbk}}[t_{i,ab}]$ on $dn^2$ commuting variables. Each entry of $f(T_1,\ldots,T_d)$ is a polynomial in ${\red{\mathcal P}}$: a word's entry is a finite sum over intermediate matrix indices of the ordered products of entries, and the empty word contributes the identity. Thus the order of the free letters is preserved even though their matrix-entry coordinates commute.

Let ${\red{\mathfrak a}}$ be the ideal generated by the finitely many $(r+1)$-minors of this generic value. Evaluation at the given ${\red{\mathbb{L}}}$-tuple is a unital map ${\red{\mathcal P}}\to {\red{\mathbb{L}}}$ that kills these minors, by the rank criterion. Hence ${\red{\mathfrak a}}$ is proper. The ideal fact just proved supplies a common ${\red{\Bbbk}}$-point. Specializing the $T_i$ at this point gives matrices over ${\red{\Bbbk}}$ whose evaluated value has all these minors zero; the rank criterion gives rank at most $r$. No finiteness or algebraicity condition on ${\red{\mathbb{L}}}/{\red{\Bbbk}}$ was used.
\end{proof}

\begin{proof}[Proof of Theorem~\ref{thm:main}]
Fix the original ${\red{\Bbbk}},d,f$. If the scalar-zero branch of Lemma~\ref{lem:reduction} occurs, $r_f(1)=0$, so the required infimum is zero. Otherwise use its normalized polynomial $h$ and fixed integers $D,A,M$, with $D\geq2$ and $A,M\geq1$.

For a positive integer $T$, Proposition~\ref{prop:finite} supplies ${\red{\mathbb{F}}}_T,p_T,Z_T$ with~\eqref{eq:precision}. In Proposition~\ref{prop:compression} take
\[
 U=\DD,\qquad V=Z_T,\qquad H=h,\qquad R=A,\qquad S=T.
\]
The symbol $\DD$ belongs to the common lattice and has order $1\leq A$, while $\ord Z_T<A$. The field has the required commuting derivations fixing ${\red{\Bbbk}}$, the degree is $D$, and the residual has order at most $-T$. Hence there are matrices over ${\red{\mathbb{F}}}_T$ of size
\[
 n_T=p_TT^2\geq1
\]
at which the rank of $h$ is at most $q_T=4DAp_TT$. Set $r_T=\min(n_T,q_T)$, an integer in $[0,n_T]$. The same value has rank at most $r_T$. Lemma~\ref{lem:descent} therefore gives matrices $X'_T,Y'_T\in M_{n_T}({\red{\Bbbk}})$ with
\[
 \rank_{\red{\Bbbk}} h(X'_T,Y'_T)\leq r_T\leq4DAp_TT.
\]
Only this finite rank condition is descended; the differential field and its symbolic equations need not descend.

For $1\leq i\leq d$, set $W_{i,T}=X'_T(Y'_T)^i$. These are permitted matrices in the original minimum, and Lemma~\ref{lem:reduction} gives
\[
 0\leq\frac{r_f(n_T)}{n_T}
 \leq\frac{\rank_{\red{\Bbbk}} f(W_{1,T},\ldots,W_{d,T})}{n_T}
 =\frac{\rank_{\red{\Bbbk}} h(X'_T,Y'_T)}{n_T}
 \leq\frac{4DA}{T}.
\]
The integers $D,A$ depend only on the fixed input and its reduction. For every real $\varepsilon>0$, choose an integer $T>4DA/\varepsilon$. The displayed ratio is then less than $\varepsilon$. All possible normalized ranks are nonnegative, so their infimum is zero. This is the exact stated infimum over positive sizes; no assertion about a finite attained zero or a bound at every size is required.
\end{proof}

\section*{Lean formalization}\label{app:formalization}
We formalize\footnote{The lean  code is available at \\ \url{https://github.com/MechMath/MakarLimanov-formalization}.} the asymptotic Makar-Limanov rank conjecture established in Theorem~\ref{thm:main} in Lean 4~\cite{moura2021lean}, using Mathlib~\cite{mathlib}. The main theorem uses only Lean's standard axioms \texttt{propext}, \texttt{Classical.choice}, and \texttt{Quot.sound}, with no \texttt{sorry} or additional axioms.
Theorem~\ref{thm:main} is formalized as following:
\begin{leancode}
theorem makarLimanov
    (K : Type*) [Field K] [IsAlgClosed K] [CharZero K](d : ℕ) : ∀ f : FreePoly K d, Nonconstant f → sInf {r : ℝ | ∃ n : ℕ, 0 < n ∧ r = (minRank f n : ℝ) / n} = 0 := by ...
\end{leancode}

\subsection*{\red{Declaration of generative AI}}
\red{The main results presented in this paper are obtained through an interaction between the authors and an artificial intelligence agent system, MechMath Agent Team (MMAT) \cite{cao2026mechmathagentteamllm}. GPT-6 Astra was used during this work to assist in exploring proof strategies and revising the exposition. The authors assume full responsibility for all proofs and mathematical details and for the mathematical correctness of the final results.}

\bibliographystyle{cas-model2-names}
\red{\bibliography{refs}}
\end{document}